\documentclass[reqno,centertags, 11pt, draft]{amsart}

\usepackage{amsmath,amsthm,amscd,amssymb}

\usepackage{latexsym}

\usepackage{mathrsfs}

\newcommand{\bbN}{{\mathbb{N}}}

\newcommand{\bbR}{{\mathbb{R}}}

\newcommand{\bbC}{{\mathbb{C}}}

\newcommand{\ac}{\textrm{ac}}

\newcommand{\sing}{\textrm{sing}}

\newcommand{\no}{\nonumber}

\newcommand{\wti}{\widetilde}

\newcommand{\supp}{\text{\rm{supp}}}

\newcommand{\beq}{\begin{equation}}

\newcommand{\eeq}{\end{equation}}

\newcommand{\ba}{\begin{align}}

\newcommand{\ea}{\end{align}}

\DeclareMathOperator{\Ima}{Im}

\allowdisplaybreaks

\numberwithin{equation}{section}

\newtheorem{theorem}{Theorem}[section]

\newtheorem{lemma}[theorem]{Lemma}

\theoremstyle{definition}

\newtheorem{definition}[theorem]{Definition}

\theoremstyle{remark}

\newtheorem{remark}{Remark}[section]

\date{}
\begin{document}
\title{Stability of Asymptotics of Christoffel-Darboux Kernels}
\author{Jonathan Breuer$^{1,3}$, Yoram Last$^{1,3}$ and Barry Simon$^{2,3}$}

\thanks{$^1$ Institute of Mathematics, The Hebrew University, 91904 Jerusalem, Israel.
E-mail: jbreuer@math.huji.ac.il; ylast@math.huji.ac.il. Supported in part by The Israel Science
Foundation (Grant No.\ 1105/10)}

\thanks{$^2$ Mathematics 253-37, California Institute of Technology, Pasadena, CA 91125, USA.
E-mail: bsimon@caltech.edu. Supported in part by NSF Grant No.\ DMS-0968856}

\thanks{$^3$ Research supported in part
by Grant No.\ 2010348 from the United States-Israel Binational Science
Foundation (BSF), Jerusalem, Israel}

\maketitle
\sloppy
\begin{abstract}
We study the stability of convergence of the Christoffel-Darboux kernel, associated with a compactly supported measure, to the sine kernel, under perturbations of the Jacobi coefficients of the measure. We prove stability under variations of the boundary conditions and stability in a weak sense under $\ell^1$ and random $\ell^2$ diagonal perturbations. We also show that convergence to the sine kernel at $x$ implies that $\mu(\{x\})=0$.
\end{abstract}
\section{Introduction}

Let $d\mu(x)=w(x)dx+d\mu_{\sing}(x)$ be a compactly supported probability measure with an infinite number of points in its support, where $\mu_\sing$ denotes the part of $\mu$ which is singular with respect to Lebesgue measure. Let $\{p_n\}_{n=0}^\infty$ be the normalized orthogonal polynomials for $d\mu$ and let $\{a_n,b_n\}_{n=1}^\infty$ be the Jacobi parameters defined by 
\beq \label{recurrence}
\begin{split}
xp_n(x) &=a_{n+1}p_{n+1}(x)+b_{n+1}p_n(x)+a_np_{n-1}(x), \quad n\geq 1 \\
xp_0(x)&=a_{1}p_{1}(x)+b_1p_0(x),
\end{split}
\eeq
and satisfying $a_n>0$, $b_n \in \bbR$, $\sup_n \left(a_n+|b_n|\right)<\infty$ (note $p_0(x) \equiv 1$ by the normalization). 

The $n$th Christoffel-Darboux (CD) kernel associated with $\mu$, $K_n(\mu;x,y)$, is the kernel of the projection from $L^2(d\mu)$ to the subspace spanned by $\{1,x,x^2,\ldots, x^{n-1} \}$. Namely,
\beq \label{CD}
K_n(\mu;x,y)=\sum_{j=0}^{n-1} p_j(x)p_j(y).
\eeq
The asymptotics of $K_n(x,y)$ for $x-y\sim \frac{1}{n}$ as $n \rightarrow \infty$ has been a topic of intensive study recently, motivated in part by the connection to the asymptotic behavior of zeros of $p_n$ (see \cite{als,last-simon-CPAM, LeLu1, totik}), and to the problem of universality in random matrix theory (see e.g., \cite{deift, Kuijlaars, lubinskyRev}). In particular, the limit
\beq \label{eq:quasiuniversality}
\lim_{n \rightarrow \infty}\frac{K_n\left(x_0+\frac{a}{n},x_0+\frac{b}{n} \right)}{n}=\frac{\sin\left(\pi \rho(x_0) (b-a) \right)}{\pi w(x_0) (b-a)},
\eeq
has been shown to hold for large classes of measures $\mu$, whenever $x_0$ is a Lebesgue point of $\mu$  (\cite{als, findley, freud, LeLu, lubinsky1, lubinskyUniv2, Simon-ext, totik} is a very partial list of relevant references). In \eqref{eq:quasiuniversality}, $\rho(x_0)$ is some positive number. In all known examples, $\rho$ is the density (i.e.\ the  derivative with respect to Lebesgue measure) of the weak limit of the sequence $\frac{K_n(x,x)}{n}d\mu(x)$. We will want to avoid such a restriction below.

As is well known \cite{simon-szego, szego}, there is a one to one correspondence (through \eqref{recurrence}) between compactly supported probability measures with infinite support and bounded real sequences $\{a_n,b_n\}_{n=1}^\infty$ satisfying $a_n>0$ for all $n$. Given a measure, $\mu$, with Jacobi parameters $\{a_n,b_n\}_{n=1}^\infty$, and a perturbing sequence $\{\beta_n\}_{n=1}^\infty$, it is natural to ask what properties of $\mu$ carry over to the measure $\mu_\beta$ associated with the Jacobi parameters  $\{a_n,b_n+\beta_n\}_{n=1}^\infty$. 

The purpose of this paper is to study the stability of \eqref{eq:quasiuniversality} under such perturbations. We shall focus on points where $\mu$ has some regularity. More precisely,

\begin{definition} \label{def:strong-Lebesgue}
We say $x_0$ is a strong Lebesgue point for $\mu$ if the following conditions hold:\\
$(i)$ Letting $F_\mu(z)=\int\frac{d\mu(t)}{t-z}$ be the Stieltjes transform of $\mu$, 
\beq \label{strong-Lebesgue0}
F_\mu(x_0+i0)=\lim_{\varepsilon \rightarrow 0+} F_\mu(x_0+i\varepsilon)
\eeq exists and is finite. \\
$(ii)$ The derivative of $\mu$ with respect to Lebesgue measure exists and is positive at $x_0$, namely 
\beq \label{strong-Lebesgue0.5}
\lim_{\varepsilon \rightarrow 0}\frac{\mu(x_0-\varepsilon,x_0+\varepsilon)}{2\varepsilon}=w(x_0)>0.
\eeq
Moreover, $x_0$ is a Lebesgue point of $w$: 
\beq \label{strong-Lebesgue1}
\lim_{\varepsilon \rightarrow 0^+}\int_{x_0-\varepsilon}^{x_0+\varepsilon}\frac{ \left|w(t)-w(x_0)\right|}{2 \varepsilon}dt=0.
\eeq  
\end{definition} 
\begin{remark}
Note that \eqref{strong-Lebesgue0.5} and \eqref{strong-Lebesgue1} imply immediately that the part of $\mu$ that is singular with respect to Lebesgue measure satisfies 
\beq \label{strong-Lebesgue2}
\lim_{\varepsilon \rightarrow 0} \frac{\mu_\sing \left(x_0-\varepsilon,x_0+\varepsilon \right)}{2 \varepsilon}=0.
\eeq 
Maximal function methods \cite{rudin} show that almost every $x_0$ w.r.t.\ $\mu_{\textrm{ac}}$(=the part of $\mu$ that is absolutely continuous with respect to Lebesgue measure) satisfies \eqref{strong-Lebesgue0.5} and \eqref{strong-Lebesgue1}. Similar methods also show that Lebesgue almost every $x_0$ satisfies \eqref{strong-Lebesgue0} (see \cite[Theorem I.4]{simon-rankone}). Thus, almost every point with respect to $\mu_{\textrm{ac}}$ is a strong Lebesgue point of $\mu$. 
\end{remark}

\begin{definition} \label{def:quasiuniversality}
We say that \emph{quasi bulk universality} (or simply \emph{quasi universality}) holds for $\mu$ at $x_0 \in \bbR$ if $x_0$ is a strong Lebesgue point of $\mu$ and  \eqref{eq:quasiuniversality} holds uniformly for $a,b$ in compact subsets of $\bbC$, for some positive number $\rho(x_0)$.
\end{definition}
\begin{remark} \label{rem:weakuniversalityequiv}
Using the uniform convergence on compacts and the continuity of the function $\frac{\sin(x-y)}{x-y}$, it is not hard to see that Definition \ref{def:quasiuniversality} is equivalent to the following two conditions (given that $x_0$ is a strong Lebesgue point): \\
i) Uniformly for $a,b \in$ compact subsets of $\bbC$
\beq \label{eq:weakuniversality1.1}
\lim_{n \rightarrow \infty}\frac{K_n\left(x_0+\frac{a}{w(x_0)K_n(x_0,x_0)},x+\frac{b}{w(x_0)K_n(x_0,x_0)} \right)}{K_n(x_0,x_0)}=\frac{\sin\left(\pi  (b-a) \right)}{\pi (b-a)},
\eeq  
(known as \emph{weak bulk universality}). \\
ii) $\lim_{n \rightarrow \infty} \frac{K_n(x_0,x_0)w(x_0)}{n}=\rho(x_0)$.
\end{remark}
\begin{remark}
In the case that $\frac{K_n(x,x)}{n}d\mu(x)$ has a weak limit, $\nu$ (aka `the density of states' or density of zeros of $p_n$ \cite{simon-DMJ}), and if $\rho(x_0)$ is the density of $\nu$ at the point $x_0$, quasi bulk universality implies bulk universality \cite{als}. 
\end{remark}

In a sense, the simplest nontrivial perturbing sequence, $\{\beta_n\}_{n=1}^\infty$, is a sequence satisfying $\beta_n=0$ for all $n \neq 1$. In order to treat this case, we first consider the problem (which is interesting in its own right) of the consequences of universality for the second kind CD kernel.

For this, recall that the \emph{second kind} orthogonal polynomials associated with $\mu$, $\{q_n\}_{n=0}^\infty$, are defined by  
\beq \label{second-kind} 
q_n(x)=\int\frac{p_{n}(x)-p_{n}(t)}{x-t}d\mu(t).
\eeq
Note that $q_n$ is a polynomial of degree $(n-1)$ for $n \geq 1$, and $q_0=0$. Moreover (see Section 2 below), $a_1 q_n$ are the orthonormal polynomials with respect to the measure $\mu'$ whose Jacobi coefficients are $\{a_{n+1},b_{n+1}\}_{n=1}^\infty$. We let 
\beq \no
\wti{K}_n(x,y)=\sum_{j=0}^{n-1} q_j(x) q_j(y)
\eeq 
be the \emph{second kind} CD kernel and we let $\wti{\mu}$ be the orthogonality measure of the second kind orthogonal polynomials (so $\wti{\mu}=a_1^2 \mu'$). We can now state our first main result, which we shall prove in Section 3.

\begin{theorem}\label{thm:secondkind}
Assume that $\mu$ has compact support and $x_0$ is a strong Lebesgue point of $\mu$. Assume further that 
\beq \label{universality2}
\lim_{n \rightarrow \infty}\frac{K_n\left(x_0+\frac{a}{n},x_0+\frac{b}{n} \right)}{n}=\frac{\sin\left(\pi \rho(x_0)(b-a) \right)}{\pi w(x_0)(b-a)}
\eeq
uniformly for $a,b$ in compact subsets of $\bbC$, for some positive number $\rho(x_0)$.

Then 
\beq \label{universality-perturbed00}
\lim_{n \rightarrow \infty}\frac{\wti{K}_n\left(x_0+\frac{a}{n},x_0+\frac{b}{n} \right)}{n}=\frac{\sin\left(\pi \rho(x_0)(b-a) \right)}{\pi \wti{w}(x_0)(b-a)}
\eeq
uniformly for $a,b$ in compact subsets of $\bbC$, where $\wti{w}(x_0) \neq 0$ is the weight of $d\wti{\mu}$ at $x_0$. 
\end{theorem}

\begin{remark}
There is an extensive literature, going  back to Kato, on the stability of absolutely continuous spectrum of Schr\"odinger operators and Jacobi matrices \cite{agmon, birman, bl, ck, deiftkillip, denisov, enss, ks, kato1, kato2, killipsimon, kuroda1, kuroda2, last, pearson}. The asymptotics of $K_n(x,y)$ for $x-y\sim \frac{1}{n}$ as $n \rightarrow \infty$ are connected to the microscopic behavior of zeros of $p_n$ which are eigenvalues of truncated operators (for example, universality implies clock behavior \cite{LeLu1, Simon-CD}) and so stability of universality is a delicate issue. To the best of our knowledge, the current paper is the first work to deal with the issue of stability of these asymptotics.   
\end{remark}

Now, let $\beta_1 \in \bbR$ and let $\mu^{(\beta_1)}$ be the orthogonality measure whose Jacobi parameters are $\{a_n, b_n+\beta_1 \delta_{n1}\}_{n=1}^\infty$. Denote the corresponding orthogonal polynomials by $\{p_n^{(\beta_1)}\}_{n=0}^\infty$, and the corresponding CD kernel
\beq \label{eq:betaCDkernel}
K_n^{(\beta_1)}(x,y)=\sum_{j=0}^{n-1}p_j^{(\beta_1)}(x)p_j^{(\beta_1)}(y).
\eeq
As we show in Section 3, Theorem \ref{thm:secondkind} implies

\begin{theorem} \label{thm:rankone}
Under the assumptions of Theorem \ref{thm:secondkind}, for any $\beta_1 \in \bbR$,
\beq \label{universality-perturbed}
\lim_{n \rightarrow \infty}\frac{K_n^{(\beta_1)}\left(x_0+\frac{a}{n},x_0+\frac{b}{n} \right)}{n}=\frac{\sin\left(\pi \rho(x_0)(b-a) \right)}{\pi w^{(\beta_1)}(x_0)(b-a)}
\eeq
uniformly for $a,b$ in compact subsets of $\bbC$, where $w^{(\beta_1)}(x_0) \neq 0$ is the weight of $d\mu^{(\beta_1)}$ at $x_0$.
\end{theorem} 

\begin{remark}
As we note in the proof of Theorem \ref{StabilityOfSCCDK}, the proof of Theorem \ref{thm:secondkind} yields an interesting formula for the limit of the symmetrized mixed CD kernel under the conditions of the theorem as well (see \eqref{cute-formula1} and \eqref{cute-formula2}). It is essentially this formula, together with the limit of $\widetilde{K}_n$, which is at the heart of the proof of Theorem \ref{thm:rankone}.  
\end{remark}

Theorems \ref{thm:secondkind} and \ref{thm:rankone} say that quasi universality is stable under a perturbation of $b_1$, with the additional bonus that the number $\rho(x_0)$ remains the same after the perturbation. Note that the limiting behavior of $\frac{K_n(x,x)d\mu(x)}{n}$ is stable under such a perturbation so in the case that $\rho$ is the density of this limit then this part of the stability is trivial. If we remove this restriction and allow $x_0$ to vary over a set of positive Lebesgue measure, we can also treat more general perturbations. The following two theorems are proven in Section 4.

\begin{theorem} \label{StabilityOfSCCDK}
Assume that $\mu$ has compact support with Jacobi parameters $\{a_n,b_n\}_{n=1}^\infty$ and that quasi universality holds at Lebesgue almost every $x \in A$. Let $\{\beta_k\}_{k=1}^\infty$ be a sequence of real numbers satisfying 
\beq \no
\sum_{k=1}^\infty |\beta_k| < \infty.
\eeq

Then quasi universality holds at Lebesgue a.e.\ $x \in A$ also for the measure corresponding to $\{a_n,b_n+\beta_n\}_{n=1}^\infty$.
\end{theorem}

For random perturbations we can allow slower decay:

\begin{theorem} \label{RandStabilityofSCCDK}
Assume that $\mu$ has compact support with Jacobi parameters $\{a_n,b_n\}_{n=1}^\infty$, orthogonal polynomials $\{p_n(x)\}_{n=0}^\infty$, and second kind orthogonal polynomials $\{q_n(x)\}_{n=0}^\infty$. Assume that quasi universality holds at Lebesgue almost every $x \in A$. 
Let $\left\{\beta_{\omega,k} \right\}_{k=1}^\infty$ be a sequence of independent random variables with zero mean satisfying 
\beq \label{L2}
\sum_{k=1}^\infty \mathbb{E}\left( \beta_{\omega,k}^2 \right) \left(|p_k(x)|+|p_{k-1}(x)|+|q_k(x)|+|q_{k-1}(x)| \right)^4<\infty
\eeq
for Lebesgue a.e.\ $x \in A$. 

Then, with probability one, quasi universality holds at Lebesgue a.e.\ $x \in A$ also for the measure corresponding to $\{a_n,b_n+\beta_{\omega,n} \}_{n=1}^\infty$.
\end{theorem}

\begin{remark}
We shall prove Theorems \ref{StabilityOfSCCDK} and \ref{RandStabilityofSCCDK} by showing that the existence of $\lim_{n \rightarrow \infty}{\frac{K_n(x,x)}{n}}$ and $\lim_{n \rightarrow \infty}\frac{\wti{K}_n(x,x)}{n}$, for a.e.\ $x$, is stable under the perturbations considered.
The results of \cite{als} then say that the existence of these limits for a.e.\ $x$ implies \eqref{eq:weakuniversality1.1} for a.e.\ $x$ (they prove it there for $a,b \in \bbR$, but the proof extends to $a,b \in \bbC$). Letting $\rho(x) \equiv \lim_{n \rightarrow \infty} \frac{K_n(x,x)}{n}w(x)$ we see, by Remark \ref{rem:weakuniversalityequiv}, that these conditions imply quasi universality at $x$. Theorem \ref{thm:secondkind} completes the picture to show that in fact quasi universality almost everywhere is equivalent to the a.e.\ existence of these limits.

If we assume that $\frac{K_n(x,x)}{n}d\mu(x)$ has a weak limit, $\nu$, and that $\nu$ is absolutely continuous with weight $\wti{\rho}(x)$, then this limit is stable under the perturbations of Theorem \ref{StabilityOfSCCDK} and \ref{RandStabilityofSCCDK}. However, even if we assume that $\lim_{n \rightarrow \infty}\frac{K_n(x,x) w(x)}{n}=\wti{\rho}(x)$ for a.e.\ $x \in A$, we do not know how to deduce this convergence for the CD kernel of the perturbed problem. The issue is that in general weak convergence and the existence of a pointwise limit do not guarantee that the pointwise limit coincides with the weak limit. Equality of these limits would follow, for example, if we know that $\frac{K_n(x,x)}{n}$ is uniformly bounded on an interval, but we do not want to assume this. This is the reason behind our notion of quasi universality.
\end{remark}

On physical grounds, one expects a connection between local continuity of $\mu$ at $x$ and the asymptotics of $\frac{K_n\left(x+\frac{a}{n},x+\frac{b}{n} \right)}{n}$. In particular, for all known examples where $\mu$ is absolutely continuous at $x$, universality has been shown to hold there. A significant motivating factor for this paper was the fact that absolute continuity of $\mu$ is stable under the perturbations considered above.

It is important to note, however, that universality can occur also for purely singular measures, as demonstrated in \cite{Bjat}. Unfortunately, we have nothing to say on the issue of stability of universality in such a case. Still, the work in \cite{Bjat} raises an interesting question: assuming universality holds at $x$, what can one say about the local continuity of $\mu$ there? The next theorem, proven in Section 5, says that universality at $x$ implies that $\mu$ cannot have a pure point there.

\begin{theorem} \label{no-point}
Assume $\mu$ has compact support. Fix $x \in \bbR$. If there exists a number $\rho(x) >0$ such that for any $a, b \in \bbR$,
\beq \label{universality}
\lim_{n \rightarrow \infty}\frac{K_n \left(x+\frac{a}{n},x+\frac{b}{n} \right)}{K_n(x,x)}=
\frac{\sin \left(\pi \rho(x)(b-a) \right)}{\pi \rho(x) (b-a)},
\eeq
then
\beq \label{infinite-limit}
\lim_{n \rightarrow \infty} K_n(x,x) =\infty.
\eeq
\end{theorem}

\begin{remark}
The condition \eqref{infinite-limit} is equivalent to $\mu \left(\{x\} \right)=0$.
This is because $\mu$ has a pure point at $x$ iff the operator of multiplication by $t$ on $L^2(d\mu(t))$ has an eigenvalue at $x$. Since $\{p_j(\cdot)\}_{j=0}^\infty$ is an orthonormal basis for $L^2(d\mu)$, this holds iff $\sum_{j=0}^\infty |p_j(x)|^2 < \infty$.
\end{remark}

The rest of this paper is structured as follows: Section 2 has some preliminary facts we shall need from the theory of rank one perturbations. The proofs of Theorems \ref{thm:secondkind} and \ref{thm:rankone} are given in Section 3. The proofs of Theorems \ref{StabilityOfSCCDK} and \ref{RandStabilityofSCCDK} are given in Section 4. Finally, Section 5 has the proof of Theorem \ref{no-point}.


\section{Preliminaries}

Let $\mu$ be a probability measure on $\bbR$, whose support is a compact, infinite set. Then $\mu$ is the spectral measure of the operator of multiplication by $x$ on the space $L^2(d\mu(x))$. The recursion relation \eqref{recurrence} says that in the orthonormal basis $\{p_n(\cdot) \}_{n=0}^\infty$, this operator is given by a Jacobi matrix 
\beq \label{eq:Jacobi}
J =\left(
\begin{array}{ccccc}
b_1    & a_1 & 0      & 0      & \dots \\
a_1    & b_2 & a_2    & 0      & \dots \\
0      & a_2 & b_3    & a_3    & \ddots \\
\vdots & \vdots   & \ddots & \ddots & \ddots \\
\end{array} \right).
\eeq
Thus $\mu$ is the spectral measure of the operator $J$ on $\ell^2(\bbN)$ and the vector $\delta_1=\left(\begin{array}{c} 1 \\ 0 \\ 0 \\ \vdots \\ \end{array} \right)$. 

If follows that, for $\beta_1 \in \bbR$, the measure $\mu^{(\beta_1)}$ corresponding to the Jacobi parameters $\{a_n, b_n+\beta_1 \delta_{n1}\}_{n=1}^\infty$ is also the spectral measure of the Jacobi matrix
\beq \label{eq:Jacobiperturbed1}
J^{(\beta_1)} =\left(
\begin{array}{ccccc}
b_1+\beta_1    & a_1 & 0      & 0      & \dots \\
a_1    & b_2 & a_2    & 0      & \dots \\
0      & a_2 & b_3    & a_3    & \ddots \\
\vdots & \vdots   & \ddots & \ddots & \ddots \\
\end{array} \right).
\eeq
But $J^{(\beta_1)} \psi= J \psi +\beta_1 \langle \delta_1, \psi \rangle \delta_1$, namely a rank one perturbation of $J$. Thus, we need some facts from the theory of rank one perturbations of self-adjoint operators. A comprehensive review of the relevant theory is given in \cite{simon-rankone}. Here we shall merely collect the facts we will use.

We first define the Stieltjes (aka Cauchy/Borel) transform of $\mu$ by
\beq \label{eq:stieltjes}
F_\mu(z)=\int\frac{d\mu(t)}{t-z}.
\eeq 
$F_\mu$ is analytic on $\bbC \setminus \supp(\mu)$ and has positive imaginary part on $\bbC_+=\{z \mid \Ima(z)>0 \}$. The limit $\lim_{\varepsilon \rightarrow 0^+}F(x+i \varepsilon) \equiv F(x+i0)$ exists for strong Lebesgue points of $\mu$ and is related to $\mu$ through the fact that 
\beq \label{eq:stieltjeslimit}
\frac{1}{\pi}\Ima F_\mu(x+i\varepsilon) dx \rightarrow d\mu(x)
\eeq
weakly as $\varepsilon \rightarrow 0^+$. In fact, for the absolutely continuous part we have
\beq \label{eq:stieltjeslimitac}
\frac{1}{\pi} \Ima F_\mu(x+i0) =w(x)
\eeq
whenever $x$ is a strong Lebesgue point of $\mu$, in the sense that the limit exists at $x$ and is equal to $w(x)$ there. This implies that
\beq \label{eq:stieltjeslimitac1}
\frac{1}{\pi} \Ima F_\mu(x+i0)dx=d\mu_{\textrm{ac}}(x)
\eeq
since a.e.\ point with respect to $\mu_{\textrm{ac}}$ is a strong Lebesgue point of $\mu$.

The identification 
\beq \no
F_\mu(z)=\langle \delta_1, \left(J-z\right)^{-1} \delta_1 \rangle,
\eeq 
through the spectral theorem, and the resolvent formula
\beq \no
\left( J^{(\beta_1)}-z \right)^{-1}-\left(J-z \right)^{-1}=-\beta_1 \left \langle \delta_1, \left(J^{(\beta_1)}-z \right)^{-1}\cdot \right \rangle \left(J-z \right)^{-1}\delta_1
\eeq
imply that
\beq \label{eq:stieltjesrankone}
F^{(\beta_1)}_\mu(z)=\frac{F_\mu(z)}{1+\beta_1 F_\mu(z)}
\eeq
where $F_\mu^{(\beta_1)}(z)=\int \frac{d\mu^{(\beta_1)}(t)}{t-z}$. This immediately implies
\beq \label{eq:stieltjesrankoneac}
\Ima F^{(\beta_1)}_\mu(z)=\frac{\Ima F_\mu(z)}{\left|1+\beta_1 F_\mu(z) \right|^2}=\frac{\Ima F_\mu(z)}{1+2\beta_1 \textrm{Re}(F_\mu(z))+\beta_1^2 \left| F_\mu(z) \right|^2}
\eeq
so that for $x$ a strong Lebesgue point of $\mu$,
\beq \label{eq:stieltjesrankoneac1}
\begin{split}
w^{(\beta_1)}(x)&=\frac{1}{\pi}\frac{\Ima F_\mu(x+i0)}{1+2\beta_1 \textrm{Re}(F_\mu(x+i0))+\beta_1^2 \left| F_\mu(x+i0) \right|^2} \\
&=\frac{w(x)}{1+2\beta_1 \textrm{Re}(F_\mu(x+i0))+\beta_1^2 \left| F_\mu(x+i0) \right|^2}. \\
\end{split}
\eeq

The case ``$\beta_1=\infty$'' is of particular significance. By \cite[Section 4]{GS} (also see \cite{simon-rankone}), as $\beta_1 \rightarrow \infty$, $J^{(\beta_1)}$ converges in the strong resolvent sense to the operator 
\beq \label{eq:secondkindoperator}
J^{(\infty)}=\left(
\begin{array}{ccccc}
0    & 0 & 0      & 0      & \dots \\
0   & b_2 & a_2    & 0      & \dots \\
0      & a_2 & b_3    & a_3    & \ddots \\
\vdots & \vdots   & \ddots & \ddots & \ddots \\
\end{array} \right).
\eeq
In addition, if $\wti{\mu}$ is the spectral measure of $J^{(\infty)}$ and the vector
$J \delta_1-\langle \delta_1, J \delta_1 \rangle \delta_1=\left(\begin{array}{c} 0 \\ a_1 \\ 0 \\ \vdots \\ \end{array} \right)$, then 
\beq \label{eq:secondkindstieltjes}
F_{\wti{\mu}}(z)=-\frac{1}{F_\mu(z)}
\eeq 
which means that
\beq \label{eq:secondkindstieltjes1}
\Ima F_{\wti{\mu}}(z)=\frac{\Ima F_{\mu}(z)}{\left| F_\mu(z) \right|^2}.
\eeq
The connection of this to $q_n$, the second kind polynomials associated with $\mu$, is through the fact that if $p^\infty_n$ are the orthogonal polynomials associated with the Jacobi matrix
\beq \no
\wti{J}=\left(
\begin{array}{ccccc}
 b_2 & a_2    & 0      & \dots \\
 a_2 & b_3    & a_3    & \ddots \\
0 & a_3 & b_4 & \ddots \\
\vdots & \vdots   & \ddots & \ddots & \ddots \\
\end{array} \right)
\eeq
then $q_n(x)=a_1^{-1} p_{n-1}^\infty(x)$. Since $\wti{\mu}$ is $a_1^2$ times the orthogonality measure of the $p_n^\infty$, we see that $q_n$ is precisely the $(n-1)$'th ortho\emph{normal} polynomial with respect to $\wti{\mu}$. Thus, $\wti{\mu}$ is the measure of orthogonality for the $q_n$'s. It follows from \eqref{eq:stieltjeslimitac}, \eqref{eq:stieltjeslimitac1}, and  \eqref{eq:secondkindstieltjes1} that whenever $x$ is a strong Lebesgue point of $\mu$ we may use  
\beq \label{eq:secondkindweight}
\wti{w}(x)=\frac{w(x)}{|F(x+i0)|^2}
\eeq
for the Radon-Nikodym derivative, $\wti{w}$, of $\wti{\mu}$.

An important part of our analysis rests on the fact that the limits above can be defined somewhat more generally. In fact, for any $a$ with $\Ima a >0$,
\beq \no
F(x+i0)=\lim_{n \rightarrow \infty}F\left(x+\frac{a}{n}\right)
\eeq
whenever $F(x+i0)$ exists. Also note that for such $a$,
\beq \no
\overline{F(x+i0)}=\lim_{n \rightarrow \infty} F\left( x+\frac{\overline{a}}{n} \right).
\eeq

To deduce Theorem \ref{thm:rankone} from Theorem \ref{thm:secondkind} we need to express $p_n^{(\beta_1)}$ as a linear combination of $p_n$ and $q_n$. This is possible since both sequences $p_n$ and $q_n$ satisfy the same recursion relation (with different boundary conditions). Since $q_0(x) \equiv 0$ and $p_0(x)=p_0^{(\beta_1)}(x)=1$, it is obvious that 
$p_n^{(\beta_1)}(x)=p_n(x)+\gamma q_n(x)$. By plugging this into \eqref{recurrence} for $p_n^{(\beta_1)}$ we immediately see that $\gamma=-\beta_1$ and so
\beq \label{eq:linearcomb}
p_n^{(\beta_1)}(x)=p_n(x)-\beta_1 q_n(x).
\eeq

\medskip
It will be convenient for us to write the recursion relation in matrix form: letting 
\beq \label{eq:one-step}
S_j(z)=\left(\begin{array}{cc} \frac{z-b_j}{a_j} & -\frac{1}{a_j} \\
a_j & 0 \end{array} \right)
\eeq
we see that 
\beq \label{eq:recursion-matrix}
\left(\begin{array}{c} p_n(z) \\ a_np_{n-1} (z) \end{array} \right)=S_n(z) \left(\begin{array}{c} p_{n-1}(z) \\ a_np_{n-2} (z) \end{array} \right).
\eeq
Note also that $\det \left( S_j(z) \right)=1$. The \emph{transfer matrix} is defined by 
\beq \label{eq:transfer}
\Phi_n(z)=S_n(z)S_{n-1}(z)\ldots S_1(z)
\eeq
so that 
\beq \label{eq:transfer1}
\left(\begin{array}{c} p_n(z) \\ a_np_{n-1} (z) \end{array} \right)=\Phi_n(z) \left(\begin{array}{c} 1 \\ 0 \end{array} \right),
\eeq
and
\beq \label{eq:transfer2}
\left(\begin{array}{c} -q_n(z) \\ -a_np_{n-1} (z) \end{array} \right)=\Phi_n(z) \left(\begin{array}{c} 0 \\ 1 \end{array} \right).
\eeq
Thus we see that
\beq \label{eq:transfer-form}
\Phi_n(z) = \left( \begin{array}{cc} p_n(z) & -q_n(z) \\
 a_n p_{n-1}(z) & -a_n q_{n-1}(z) \end{array} \right).
\eeq
Finally, we note that, by the fact that $\det \Phi_n(z) =1$,
\beq \label{eq:norm-transfer}
\|\Phi_n(z)\|=\|\left(\Phi_n(z) \right)^{-1} \|.
\eeq


\section{Proof of Theorems \ref{thm:secondkind} and \ref{thm:rankone}}

We first prove the following lemma:

\begin{lemma} \label{lemma-rankone}
Under the conditions of Theorem \ref{thm:secondkind}
\beq \label{mixed1}
\lim_{n \rightarrow \infty} \frac{1}{n}\left( \int \frac{K_n\left(x_0+\frac{a}{n},s \right)}{ \left(x_0+\frac{b}{n}-s \right)} d\mu(s) \right)=\int_{-\infty}^\infty \frac{\sin\pi \rho(x_0) \left(s-a \right)}{\pi (s-a)(b-s)}ds,
\eeq
\beq \label{mixed2}
\lim_{n \rightarrow \infty} \frac{1}{n}\left(\int  \frac{K_n\left(x_0+\frac{b}{n},t \right)}{\left(x_0+\frac{a}{n}-t \right) }d\mu(t) \right)=\int_{-\infty}^\infty\frac{\sin \pi \rho(x_0) \left(t-b \right)}{\pi (t-b)(a-t)}dt,
\eeq
and
\beq \label{notmixed}
\begin{split}
& \lim_{n \rightarrow \infty}\frac{1}{n}\left(\int \int \frac{K_n(t,s)}{\left(x_0+\frac{a}{n}-t \right) \left(x_0+\frac{b}{n}-s \right)}d\mu(t) d\mu(s) \right) \\
&\quad=\int_{-\infty}^\infty \int_{-\infty}^\infty \frac{\sin{\pi \rho(x_0)\left(t-s \right)}}{\pi (t-s)(t-a)(s-b)}w(x_0)dt ds
\end{split}
\eeq
for any $a,b$ with $\Ima(a)>0$, $\Ima(b)<0$.
\end{lemma}

\begin{proof}
We first prove \eqref{notmixed}. For simplicity of notation we assume $x_0=0$.

Fix $M>0$, and let $I_n=[-\frac{M}{n}, \frac{M}{n}]$. We split the integral as follows
\beq \no
\begin{split}
&\int \int \frac{K_n(t,s)d\mu(t)d\mu(s)}{\left(\frac{a}{n}-t \right) \left(\frac{b}{n}-s \right)} =\int_{I_n} \int_{I_n} \frac{K_n(t,s)d\mu_\ac(t)d\mu_\ac(s)}{\left(\frac{a}{n}-t \right) \left(\frac{b}{n}-s \right)} \\
&\quad+\int_{I_n} \int_{I_n} \frac{K_n(t,s)d\mu_\ac(t)d\mu_\sing(s)}{\left(\frac{a}{n}-t \right) \left(\frac{b}{n}-s \right)}
+\int_{I_n} \int_{I_n} \frac{K_n(t,s)d\mu_\sing(t)d\mu_\ac(s)}{\left(\frac{a}{n}-t \right) \left(\frac{b}{n}-s \right)}\\
&\quad+\int_{I_n} \int_{I_n} \frac{K_n(t,s)d\mu_\sing(t) d\mu_\sing(s)}{\left(\frac{a}{n}-t \right) \left(\frac{b}{n}-s \right)} + \int_{\bbR \setminus I_n} \int_{I_n} \frac{K_n(t,s)d\mu(t) d\mu(s)}{\left(\frac{a}{n}-t \right) \left(\frac{b}{n}-s \right)}\\
&\quad + \int_{I_n} \int_{\bbR \setminus I_n} \frac{K_n(t,s)d\mu(t) d\mu(s)}{\left(\frac{a}{n}-t \right) \left(\frac{b}{n}-s \right)}+\int_{\bbR \setminus I_n} \int_{\bbR \setminus I_n} \frac{K_n(t,s)d\mu(t) d\mu(s)}{\left(\frac{a}{n}-t \right) \left(\frac{b}{n}-s \right)}.
\end{split}
\eeq

We evaluate the first term by changing variables
\beq \no
\begin{split}
&\int_{I_n} \int_{I_n} \frac{K_n(t,s)}{\left(\frac{a}{n}-t \right) \left(\frac{b}{n}-s \right)}d\mu_\ac(t) d\mu_\ac(s) \\
&\quad=\int_{I_n} \int_{I_n} \frac{K_n(t,s)}{\left(\frac{a}{n}-t \right) \left(\frac{b}{n}-s \right)}w(t) w(s) dtds \\
&\quad = \int_{-M}^M \int_{-M}^M \frac{K_n(t/n,s/n)}{\left(a-t \right) \left(b-s \right)}w(t/n) w(s/n) dtds.
\end{split}
\eeq
By the fact that $\frac{K_n\left(\frac{a}{n},\frac{b}{n} \right)}{n} \rightarrow \frac{\sin\left(\pi \rho(0)(b-a) \right)}{\pi w(0)(b-a)}$ as $n \rightarrow \infty$, uniformly on compacts, together with the fact that $0$ is a Lebesgue point of $w$, we see that 
\beq \no
\begin{split}
& \lim_{n \rightarrow \infty}\frac{1}{n}\left(\int_{I_n} \int_{I_n} \frac{K_n(t,s)}{\left(\frac{a}{n}-t \right) \left(\frac{b}{n}-s \right)}d\mu_\ac(t) d\mu_\ac(s) \right) \\
&\quad=\int_{-M}^M \int_{-M}^M \frac{\sin{\pi \rho(0)\left(t-s \right)}}{\pi (t-s)(t-a)(s-b)}w(0)dt ds
\end{split}
\eeq
which converges to the desirable limit as  $M \rightarrow \infty$. Suppose we show that
\beq \label{remainder1}
\limsup_{n \rightarrow \infty} \frac{1}{n} \int_{\bbR \setminus I_n} \int_{I_n} \frac{K_n(t,s)}{\left(\frac{a}{n}-t \right) \left(\frac{b}{n}-s \right)}d\mu(t) d\mu(s)=O(M^{-1/2})
\eeq
\beq \label{remainder2}
\limsup_{n \rightarrow \infty} \frac{1}{n} \int_{I_n} \int_{\bbR \setminus I_n} \frac{K_n(t,s)}{\left(\frac{a}{n}-t \right) \left(\frac{b}{n}-s \right)}d\mu(t) d\mu(s)=O(M^{-1/2}),
\eeq
\beq \label{remainder3}
\limsup_{n \rightarrow \infty} \frac{1}{n} \int_{\bbR \setminus I_n} \int_{\bbR \setminus I_n} \frac{K_n(t,s)}{\left(\frac{a}{n}-t \right) \left(\frac{b}{n}-s \right)}d\mu(t) d\mu(s)=O(M^{-1/2}),
\eeq
and
\beq \label{remainder4}
\limsup_{n \rightarrow \infty} \frac{1}{n} \int_{ I_n} \int_{ I_n} \frac{K_n(t,s)}{\left(\frac{a}{n}-t \right) \left(\frac{b}{n}-s \right)}d\mu_\sing(t) d\mu(s)=0.
\eeq
Then, by taking first $n \rightarrow \infty$ and then $M \rightarrow \infty$, we are done.

For any sets $I, J \subseteq \bbR$, write
\beq \no
\begin{split}
& \left| \int_{I} \int_{J} \frac{K_n(t,s)}{\left(\frac{a}{n}-t \right) \left(\frac{b}{n}-s \right)}d\mu(t)d\mu(s) \right| \\
&\quad = \left|\sum_{j=0}^{n-1} \int_{I} \frac{p_j(t)}{\frac{a}{n}-t} d\mu(t) \int_{J} \frac{p_j(s)}{ \frac{b}{n}-s} d\mu(s) \right| \\
&\quad \leq \left(\sum_{j=0}^{n-1}\left| \int_{I} \frac{p_j(t)}{\frac{a}{n}-t} d\mu(t)\right|^2 \right)^{1/2} \left( \sum_{j=0}^{n-1}\left| \int_{J} \frac{p_j(s)}{ \frac{b}{n}-s} d\mu(s)\right|^2 \right)^{1/2} \\
&\quad \leq \left(\sum_{j=0}^{\infty}\left| \int_{I} \frac{p_j(t)}{\frac{a}{n}-t} d\mu(t)\right|^2 \right)^{1/2} \left( \sum_{j=0}^{\infty}\left| \int_{J} \frac{p_j(s)}{ \frac{b}{n}-s} d\mu(s)\right|^2 \right)^{1/2},
\end{split}
\eeq
by Cauchy-Schwarz, and note that
\beq \no 
\begin{split}
\sum_{j=0}^{\infty}\left| \int_{I} \frac{p_j(t)d\mu(t)}{\frac{z}{n}-t}\right|^2 = \sum_{j=0}^\infty \int_{I} \frac{p_j(t) d\mu(t)}{\frac{z}{n}-t}  \int_{I} \frac{p_j(s)d\mu(s)}{\frac{\overline{z}}{n}-s} = \int_{I}\frac{d\mu(t)}{\left|\frac{z}{n}-t \right|^2},
\end{split}
\eeq
for any $z$ with $\Ima(z) \neq 0$, by the completeness of $\{p_j (\cdot)\}_{j=0}^\infty$ in $L^2(d\mu)$. 

It follows that
\beq \no
 \left| \int_{\bbR \setminus I_n} \int_{ I_n} \frac{K_n(t,s) d\mu(t) d\mu(s)}{\left(\frac{a}{n}-t \right) \left(\frac{b}{n}-s \right)} \right| 
 \leq \left(\int_{\bbR \setminus I_n}\frac{d\mu(t)}{\left|\frac{a}{n}-t \right|^2} \right)^{1/2} \left( \int_{ I_n} \frac{d\mu(s)}{\left| \frac{b}{n} -s\right|^2} \right)^{1/2},
\eeq
\beq \no
\left| \int_{I_n} \int_{\bbR \setminus I_n} \frac{K_n(t,s) d\mu(t) d\mu(s)}{\left(\frac{a}{n}-t \right) \left(\frac{b}{n}-s \right)} \right|  \leq \left(\int_{ I_n}\frac{d\mu(t)}{\left|\frac{a}{n}-t \right|^2} \right)^{1/2} \left( \int_{\bbR \setminus I_n} \frac{d\mu(s)}{\left| \frac{b}{n} -s\right|^2} \right)^{1/2},
\eeq
\beq \no
\begin{split}
& \left| \int_{\bbR \setminus I_n} \int_{\bbR \setminus I_n} \frac{K_n(t,s)d\mu(t) d\mu(s)}{\left(\frac{a}{n}-t \right) \left(\frac{b}{n}-s \right)} \right| \\
&\quad \leq \left(\int_{\bbR \setminus I_n}\frac{d\mu(t)}{\left|\frac{a}{n}-t \right|^2} \right)^{1/2} \left(  \int_{\bbR \setminus  I_n} \frac{d\mu(s)}{\left| \frac{b}{n} -s\right|^2} \right)^{1/2},
\end{split}
\eeq
and (by further restricting from $I_n$ to a supporting set of zero Lebesgue measure for $\mu_\sing$)
\beq \no
 \left|\int_{ I_n} \int_{ I_n} \frac{K_n(t,s) d\mu_\sing(t) d\mu(s)}{\left(\frac{a}{n}-t \right) \left(\frac{b}{n}-s \right)} \right| \leq  \left(\int_{I_n}\frac{d\mu_\sing(t)}{\left|\frac{a}{n}-t \right|^2} \right)^{1/2} \left( \int_{ I_n} \frac{d\mu(s)}{\left| \frac{b}{n} -s\right|^2} \right)^{1/2}.
\eeq

Thus, if we show that
\beq \label{first_estimate}
\int_{I_n} \frac{d\mu(t)}{\left|\frac{a}{n}-t \right|^2}=O( n ),
\eeq
\beq \label{second_estimate}
\int_{I_n} \frac{d\mu_\sing(t)}{\left|\frac{a}{n}-t \right|^2}=o(n),
\eeq
and 
\beq \label{third_estimate}
\int_{\bbR \setminus I_n} \frac{d\mu(t)}{\left|\frac{a}{n}-t \right|^2}=O\left(\frac{n}{M} \right)
\eeq
for any $a \in \bbC$ with $\Ima(a) \neq 0$, then \eqref{remainder1}--\eqref{remainder4} will follow. 

Note that $x_0$ being a strong Lebesgue point of $\mu$ implies that for sufficiently large $n$, $\mu\left( \left[-\frac{M}{n}, t \right] \right) \lesssim  t+\frac{M}{n}$ for any $-\frac{M}{n}< t \leq \frac{M}{n}$. Regarding $\mu_\sing$, for any $\varepsilon >0$, for $n$ large enough (depending on $\varepsilon$), $\mu_\sing\left( \left[-\frac{M}{n}, t \right] \right) \leq \varepsilon \left( t+\frac{M}{n} \right)$ for any $-\frac{M}{n} < t \leq \frac{M}{n}$. Thus,  integration by parts gives
\beq \no
\begin{split}
\int_{I_n} \frac{d\mu(t)}{\left|\frac{a}{n}-t \right|^2}&\lesssim n^2 \mu(I_n)\frac{1}{|a-M|^2}+\left| \int_{I_n}\frac{\mu\left( \left[-\frac{M}{n}, t \right] \right)}{\left|\frac{a}{n}-t \right|^3}dt \right| \\
& \lesssim \frac{nM}{|a-M|^2}+\int_{I_n}\frac{\left|t +\frac{M}{n} \right|}{\left|t-\frac{a}{n} \right|^3}dt \\
& \leq \frac{nM}{|a-M|^2}+n \int_{-M}^M\frac{\left|t +M \right|}{\left|t-a \right|^3}dt \\
&=\frac{nM}{|a-M|^2}+n \int_{0}^{2M}\frac{t }{\left|t-M-a \right|^3}dt  \\
&\lesssim \frac{nM}{|a-M|^2}+n \left(\frac{2M}{|M-a|^2}\right)+n \int_{0}^{2M}\frac{1 }{\left|t-M-a \right|^2}dt\\
&\leq \frac{nM}{|a-M|^2}+n \left(\frac{2M}{|M-a|^2}\right)+n \int_{-\infty}^{\infty}\frac{1 }{\left|t-a \right|^2}dt  =O(n)
\end{split}
\eeq
where the implicit constant is \emph{independent of $M$}. This is \eqref{first_estimate}. 

In the same way,
\beq \no
\begin{split}
\int_{I_n} \frac{d\mu_\sing(t)}{\left|\frac{a}{n}-t \right|^2}&\lesssim n^2 \mu_\sing(I_n)\frac{1}{|a-M|^2}+\left| \int_{I_n}\frac{\mu_\sing\left( \left[-\frac{M}{n}, t \right] \right)}{\left|\frac{a}{n}-t \right|^3}dt \right| \\
& \lesssim \varepsilon \frac{nM}{|a-M|^2}+\varepsilon\int_{I_n}\frac{\left|t +\frac{M}{n} \right|}{\left|t-\frac{a}{n} \right|^3}dt \\
& \leq \varepsilon \frac{nM}{|a-M|^2}+\varepsilon n \int_{-M}^M\frac{t +M }{\left|t-a \right|^3}dt \\
&=  \varepsilon \frac{nM}{|a-M|^2}+\varepsilon n \int_{0}^{2M}\frac{t }{\left|t-M-a \right|^3}dt = O(n)\varepsilon
\end{split}
\eeq
which means that 
\beq \no
\limsup_{n \rightarrow \infty} \frac{1}{n} \int_{I_n} \frac{d\mu_\sing(t)}{\left|\frac{a}{n}-t \right|^2} \lesssim \varepsilon
\eeq
for $\varepsilon$ arbitrarily small. This is \eqref{second_estimate}.

As for \eqref{third_estimate}, following \cite{bs} we define $H_n=[-\frac{M}{n^{1/3}},\frac{M}{n^{1/3}}]$ and split the integral over $H_n$. For $M \geq 2|a|$  
\beq \no
\int_{\bbR \setminus I_n} \frac{d\mu(t)}{\left|\frac{a}{n}-t \right|^2} \leq 4 \int_{\bbR \setminus I_n} \frac{d\mu(t)}{t^2}=4 \int_{\bbR \setminus H_n} \frac{d\mu(t)}{t^2}+4 \int_{H_n \setminus I_n} \frac{d\mu(t)}{t^2}.
\eeq
Since $\mu$ is a probability measure and $t \notin H_n$ satisfies $ t^2 \geq \frac{M^2}{n^{2/3}}$ we see that
\beq \no
 \int_{\bbR \setminus H_n} \frac{d\mu(t)}{t^2} \leq \frac{n^{2/3}}{M^2}.
\eeq
For the remaining integral we use (again) integration by parts:
\beq \no
\begin{split}
 \int_{H_n \setminus I_n} \frac{d\mu(t)}{t^2} & \lesssim  \frac{n^{2/3}}{M^2}\mu\left(H_n \setminus I_n \right)+\int_{M/n}^{M/n^{1/3}}\frac{\mu \left(\left[\frac{M}{n},t \right] \right)+\mu \left(\left[-t,-\frac{M}{n}\right] \right)}{t^3}dt\\
& \lesssim \frac{M}{n^{1/3}} \frac{n^{2/3}}{M^2}+\int_{M/n}^{M/n^{1/3}}\frac{t-\frac{M}{n}}{t^3}dt \lesssim \frac{n}{M}+n\int_{M}^{n^{2/3}M}\frac{t-M}{t^3}dt \\
&\lesssim \frac{n}{M}+n\int_{M}^{\infty}\frac{t-M}{t^3}dt=O\left( \frac{n}{M} \right). 
\end{split}
\eeq
The last two inequalities imply \eqref{third_estimate} and thus finish the proof of \eqref{notmixed}.

The proof of \eqref{mixed1} and \eqref{mixed2} follows the same strategy of the proof of \cite[Theorem 3.1]{bs} with the following modifications: first, $K_n(x,x)w(x)$ of that proof (denoted in \cite{bs} by $\wti{K}_n(x,x)$) needs to be replaced by $n$ and the appropriate modifications made to the limit. For this purpose note that, by \eqref{universality2}, $\lim\frac{K_n(x_0,x_0)w(x_0)}{n}=\rho(x_0)$. Second, the condition of the measure $\mu$ being purely absolutely continuous in a neighborhood of $0$ may be relaxed to the conditions satisfied by $\mu$ here. The appropriate changes to the proof proceed by using integration by parts arguments in much the same way as we did above. 
\end{proof}

\begin{proof}[Proof of Theorem \ref{thm:secondkind}]
We first prove \eqref{universality-perturbed00} for $a ,b$ satisfying $\Ima(a)>0$, $\Ima(b)<0$. Write
\beq \no
\begin{split}
& \frac{\wti{K}_n\left(x_0+\frac{a}{n},x_0+\frac{b}{n} \right)}{n}=\frac{1}{n}\sum_{j=0}^{n-1}q_j\left(x_0+\frac{a}{n} \right)q_{j} \left(x_0+\frac{b}{n}\right)\\
&\quad =\frac{1}{n}\sum_{j=0}^{n-1}\left( \int\frac{p_j\left(x_0+\frac{a}{n} \right)-p_j(t)}{x_0+\frac{a}{n}-t}d\mu(t) \int\frac{p_{j}\left(x_0+\frac{b}{n} \right)-p_{j}(s)}{x_0+\frac{b}{n}-s}d\mu(s) \right)\\
&\quad=A_n+B_n+C_n
\end{split}
\eeq
where $A_n, B_n, C_n$ are obtained by carrying out the multiplication and collecting the terms, so
\beq \no
\begin{split}
&A_n=\frac{1}{n}\sum_{j=0}^{n-1}\left( p_j\left(x_0+\frac{a}{n} \right)p_{j}\left(x_0+\frac{b}{n} \right)\left(\int\frac{d\mu(t)}{x_0+\frac{a}{n}-t} \int\frac{d\mu(s)}{x_0+\frac{b}{n}-s} \right) \right)\\
&\quad=\frac{K_n\left(x_0+\frac{a}{n},x_0+\frac{b}{n} \right)}{n}\left(\int\frac{d\mu(t)}{x_0+\frac{a}{n}-t} \int\frac{d\mu(s)}{x_0+\frac{b}{n}-s} \right),
\end{split}
\eeq
\beq \no
B_n=\frac{1}{n}\left(\int \int \frac{K_n(t,s)}{\left(x_0+\frac{a}{n}-t \right) \left(x_0+\frac{b}{n}-s \right)}d\mu(t) d\mu(s) \right),
\eeq
and
\beq \no
C_n=\frac{-1}{n}\left(\int \int \frac{K_n\left(x_0+\frac{a}{n},s \right)+K_n\left(x_0+\frac{b}{n},t \right)}{\left(x_0+\frac{a}{n}-t \right) \left(x_0+\frac{b}{n}-s \right)}d\mu(t) d\mu(s) \right).
\eeq

By \eqref{universality2}, the fact that $x_0$ is a strong Lebesgue point, and the fact that $\Ima(a) \Ima(b)<0$ we get
\beq \no
\lim_{n \rightarrow \infty} A_n=\frac{\sin\left(\pi \rho(x_0)(b-a) \right)}{\pi w(x_0)(b-a)}\left|F(x+i0) \right|^2=\frac{\sin\left(\pi \rho(x_0)(b-a) \right)}{\pi \wti{w}(x_0)(b-a)}.
\eeq
Where we used \eqref{eq:secondkindweight} to write $\wti{w}(x)=\frac{w(x_0)}{\left|F(x_0+i0) \right|^2}$.

By \eqref{notmixed}, 
\beq \no
\begin{split}
&\lim_{n \rightarrow \infty}B_n=\int_{-\infty}^\infty \int_{-\infty}^\infty \frac{\sin \left(\pi \rho(x_0)\left(t-s \right) \right)}{\pi (t-s)(t-a)(s-b)}w(x_0)dt ds \\
&\quad=\int_{-\infty}^\infty \int_{-\infty}^\infty \frac{\sin \left(\pi \rho(x_0)\left(t-s \right) \right)}{\pi (t-s)(t-a)(s-b)}dt ds \frac{\Ima \left(F(x_0+i0) \right)}{\pi} \\
&\quad=\int_{-\infty}^\infty \int_{-\infty}^\infty \frac{\sin \left(\pi \rho(x_0)\left(t-s \right) \right)}{\pi^2 (t-s)(t-a)(s-b)}dt ds \left(\frac{F(x_0+i0)-\overline{F(x_0+i0)}}{2i}\right),
\end{split}
\eeq
and by \eqref{mixed1} and \eqref{mixed2}, together with 
\beq \no
\lim_{n \rightarrow \infty}\int\frac{d\mu(t)}{x_0+\frac{b}{n}-t}=F(x_0+i0)=\overline{\lim_{n \rightarrow \infty}\int\frac{d\mu(t)}{x_0+\frac{a}{n}-t}},
\eeq
we see that
\beq \no 
\begin{split}
&\lim_{n \rightarrow \infty}C_n \\
&\quad=\int_{-\infty}^\infty \frac{\sin \left(\pi \rho(x_0) (s-a) \right)\overline{F(x_0+i0)}+\sin \left(\pi\rho(x_0)(s-b) \right)F(x_0+i0)}{\pi (s-a)(s-b)}ds.
\end{split}
\eeq

Combining the limiting expressions for $A_n$, $B_n$ and $C_n$, we see that
\beq \no 
\begin{split}
& \frac{\wti{K}_n\left(x_0+\frac{a}{n},x_0+\frac{b}{n} \right)}{n}-\frac{\sin\left(\pi \rho(x_0)(b-a) \right)}{\pi \wti{w}(x_0)(b-a)} \\
&\quad =\int_{-\infty}^\infty \frac{\sin \left(\pi \rho(x_0) (s-a) \right)\overline{F(x_0+i0)}+\sin \left(\pi\rho(x_0)(s-b) \right)F(x_0+i0)}{\pi (s-a)(s-b)}ds \\
&\qquad+\int_{-\infty}^\infty \int_{-\infty}^\infty \frac{\sin \left(\pi \rho(x_0)\left(t-s \right) \right)}{\pi^2 (t-s)(t-a)(s-b)}dt ds \left(\frac{F(x_0+i0)-\overline{F(x_0+i0)}}{2i}\right)\\
&\qquad+o(1).
\end{split}
\eeq
This step of the proof will therefore be complete if we show that
\beq \no
\begin{split}
\int_{-\infty}^\infty \frac{\sin \left(\pi \rho(x_0) (s-a) \right)}{\pi (s-a)(s-b)}ds&=-\int_{-\infty}^\infty \frac{\sin \left(\pi\rho(x_0)(s-b) \right)}{\pi (s-a)(s-b)}ds \\
&=\int_{-\infty}^\infty \int_{-\infty}^\infty \frac{\sin \left(\pi \rho(x_0)\left(t-s \right) \right)}{2\pi^2 i (t-s)(t-a)(s-b)}dt ds.
\end{split}
\eeq

We first write
\beq \no
\begin{split}
&\int_{-\infty}^\infty \int_{-\infty}^\infty \frac{\sin \left(\pi \rho(x_0)\left(t-s \right) \right)}{2\pi^2 i (t-s)(t-a)(s-b)}dt ds \\
&\quad =\int_{-\infty}^\infty \frac{ds}{2\pi i (s-b)}\int_{-\infty}^\infty \frac{\sin \left(\pi \rho(x_0)\left(t-s \right) \right)}{\pi (t-s)(t-a)}dt. 
\end{split}
\eeq
Now, the inner integral can be evaluated using contour integration, by first deforming $\bbR$ into a path, $\Gamma$, which differs from $\bbR$ only by bypassing $s$ along a small semicircle through the lower half-plane around $s$. We then split the integrand as follows:
\beq \no
\begin{split}
&\int_{-\infty}^\infty \frac{\sin \left(\pi \rho(x_0)\left(t-s \right) \right)}{\pi (t-s)(t-a)}dt =
\int_{\Gamma} \frac{\sin \left(\pi \rho(x_0)\left(t-s \right) \right)}{\pi (t-s)(t-a)}dt \\
&\quad=
\int_{\Gamma} \frac{e^{i \left(\pi \rho(x_0)\left(t-s \right) \right)}-e^{-{i \left(\pi \rho(x_0)\left(t-s \right) \right)}}}{2\pi i (t-s)(t-a)}dt\\
&\quad=\frac{1}{2\pi i}\int_{\Gamma} \frac{e^{i \left(\pi \rho(x_0)\left(t-s \right) \right)}}{(t-s)(t-a)}dt-\frac{1}{2 \pi i}\int_{\Gamma} \frac{e^{-{i \left(\pi \rho(x_0)\left(t-s \right) \right)}}}{ (t-s)(t-a)}dt.
\end{split}
\eeq
The first integral is evaluated by contour integration through the upper half-plane to show:
\beq \no
\frac{1}{2\pi i}\int_{\Gamma} \frac{e^{i \left(\pi \rho(x_0)\left(t-s \right) \right)}}{(t-s)(t-a)}dt= \frac{e^{i \pi \rho(x_0) \left(a-s \right)}}{(a-s)}+\frac{1}{\left(s-a \right)}=\frac{e^{i \pi \rho(x_0) \left(a-s \right)}-1}{\left(a-s \right)},
\eeq
and the second integral is evaluated through the lower half-plane to show:
\beq \no
\frac{1}{2 \pi i}\int_{\Gamma} \frac{e^{-{i \left(\pi \rho(x_0)\left(t-s \right) \right)}}}{ (t-s)(t-a)}dt=0.
\eeq
Thus we see that
\beq \no
\int_{-\infty}^\infty \int_{-\infty}^\infty \frac{\sin \left(\pi \rho(x_0)\left(t-s \right) \right)}{2\pi^2 i (t-s)(t-a)(s-b)}dt ds 
=\int_{-\infty}^\infty \frac{e^{i \pi \rho(x_0) \left(a-s \right)}-1}{2\pi i (s-b)\left(a-s \right)}ds.
\eeq
Now note that
\beq \no
\begin{split}
&\int_{-\infty}^\infty \frac{e^{i \pi \rho(x_0) \left(a-s \right)}-1}{2\pi i (s-b)\left(a-s \right)}ds-
\int_{-\infty}^\infty \frac{\sin \left(\pi \rho(x_0) (s-a) \right)}{\pi (s-a)(s-b)}ds \\
&\quad =\int_{-\infty}^\infty \frac{e^{i \pi \rho(x_0) \left(a-s \right)}-1}{2\pi i (s-b)\left(a-s \right)}ds-
\int_{-\infty}^\infty \frac{\sin \left(\pi \rho(x_0) (a-s) \right)}{\pi (a-s)(s-b)}ds \\
&\quad =\int_{-\infty}^\infty \frac{e^{i \pi \rho(x_0) \left(a-s \right)}-1}{2\pi i (s-b)\left(a-s \right)}ds-
\int_{-\infty}^\infty \frac{e^{i \left(\pi \rho(x_0) (a-s) \right)}-e^{-i \left(\pi \rho(x_0) (a-s) \right)}}{2 \pi i (a-s)(s-b)}ds\\
&\quad =
\int_{-\infty}^\infty \frac{e^{-i \left(\pi \rho(x_0) (a-s) \right)}-1}{2 \pi i (a-s)(s-b)}ds=0
\end{split}
\eeq
by contour integration through the upper half-plane! (Note that $a$ is \emph{not} a pole of the integrand and the integrand decays like $|s|^{-2}$ as $|s| \rightarrow \infty$ in the upper half plane).

By writing
\beq \no
\begin{split}
&\int_{-\infty}^\infty \int_{-\infty}^\infty \frac{\sin \left(\pi \rho(x_0)\left(t-s \right) \right)}{2\pi^2 i (t-s)(t-a)(s-b)}dt ds \\
&\quad =\int_{-\infty}^\infty \frac{dt}{2\pi i (t-a)}\int_{-\infty}^\infty \frac{\sin \left(\pi \rho(x_0)\left(s-t\right) \right)}{\pi (s-t)(s-b)}ds
\end{split}
\eeq 
and carrying out the analogous computation (essentially, interchanging the roles of ``upper half-plane'' and ``lower half-plane'' in the argument above) we see that also
\beq \no
\int_{-\infty}^\infty \frac{\sin \left(\pi\rho(x_0)(s-b) \right)}{\pi (s-a)(s-b)}ds =-\int_{-\infty}^\infty \int_{-\infty}^\infty \frac{\sin \left(\pi \rho(x_0)\left(t-s \right) \right)}{2\pi^2 i (t-s)(t-a)(s-b)}dt ds
\eeq
which finishes the first step of the proof. Since it will be important again later, we note we have shown that 
\beq \label{eq:important_equality}
\int_{-\infty}^\infty \frac{\sin \left(\pi \rho(x_0) (s-a) \right)}{\pi (s-a)(s-b)}ds=-\int_{-\infty}^\infty \frac{\sin \left(\pi\rho(x_0)(s-b) \right)}{\pi (s-a)(s-b)}ds. 
\eeq

Now, by taking $b=\overline{a}$ we see that
\beq \no
\frac{1}{n}\sum_{j=0}^{n-1}\left| q_j\left(x_0+\frac{a}{n} \right)\right|^2 
\eeq
is bounded uniformly on compact sets of $\bbC \setminus \bbR$. Since this is true also for 
\beq \no
\frac{1}{n}\sum_{j=0}^{n-1}\left| p_j\left(x_0+\frac{a}{n} \right)\right|^2,
\eeq
we deduce, using Cauchy-Schwarz and the boundedness of $a_n$, that
\beq \no
\frac{1}{n} \sum_{j=0}^{n-1} \left \| \Phi_j \left(x_0+\frac{a}{n}\right) \right \|^2
\eeq
is bounded uniformly on compact sets of $\bbC \setminus \bbR$.

It now follows from a modification of the proof of \cite[Theorem 3]{als} that in fact 
\beq \no
\sup_{n} \frac{1}{n} \sum_{j=0}^{n-1} \left \| \Phi_j \left(x_0+\frac{a}{n} \right) \right \|^2 < \infty
\eeq
is bounded uniformly for $a$ in compact subsets of $\bbC$. Explicitly, note that for any $a,b \in \bbC$, 
\beq \no
\left \| S_j\left(x_0+\frac{a}{n}\right)-S_j\left(x_0+\frac{b}{n} \right) \right \| \leq \alpha_-^{-1}\frac{|a-b|}{n},
\eeq
where $\alpha_-=\inf_n a_n >0$ (which follows from \cite{dombrowski} since $\mu$ has a non-trivial absolutely continuous component). 
Writing 
\beq \no
\begin{split}
\Phi_j\left(x_0+\frac{a}{n} \right)^{-1}\Phi_j\left(x_0+\frac{b}{n} \right) &= \left(1+B_j\right) \left(1+B_{j-1} \right) \ldots \left(1+B_1 \right),
\end{split}
\eeq
with
\beq \no
B_k=\Phi_k\left(x_0+\frac{a}{n} \right)^{-1} \left(S_k\left(x_0+\frac{b}{n} \right)-S_k \left(x_0+\frac{a}{n} \right) \right)\Phi_{k-1}\left(x_0+\frac{a}{n} \right),
\eeq
we get that 
\beq \no
\begin{split}
& \left \|\Phi_j \left(x_0 +\frac{b}{n} \right) \right \| \\
& \leq\left \|\Phi_j \left(x_0 +\frac{a}{n} \right) \right \| \textrm{exp} \left(\frac{\alpha_-^{-1} |a-b|}{n} \sum_{k=1}^j \left \|\Phi_j \left(x_0 +\frac{a}{n} \right) \right\| \left\|\Phi_{j-1} \left(x_0 +\frac{a}{n}\right) \right\| \right) \\
& \leq\left \|\Phi_j \left(x_0 +\frac{a}{n} \right) \right \| \textrm{exp} \left(\frac{\alpha_-^{-1} |a-b|}{n} \sum_{k=1}^j \left \|\Phi_j \left(x_0 +\frac{a}{n} \right) \right\|^2  \right) \\
& \leq\left \|\Phi_j \left(x_0 +\frac{a}{n} \right) \right \| \textrm{exp} \left(\frac{\alpha_-^{-1} |a-b|}{n} \sum_{k=1}^{n-1} \left \|\Phi_j \left(x_0 +\frac{a}{n} \right) \right\|^2  \right). \\
\end{split}
\eeq
(Note that $\left \| B_k \right \| \leq \left \|\Phi_k \left(x_0 +\frac{b}{n} \right)  \right\| \left\|\Phi_{k-1} \left(x_0 +\frac{b}{n} \right)  \right \| \frac{\alpha_-^{-1}|a-b|}{n}$, by its definition and  \eqref{eq:norm-transfer}; also note that $\left\|1+B_k \right\|\leq \textrm{exp} \left(\|B_k\| \right)$).

Now, if 
\beq \no
\sup_n \frac{1}{n} \sum_{j=0}^{n-1} \left \| \Phi_j \left(x_0+\frac{a}{n}\right) \right \|^2=C,
\eeq
it follows that
\beq \no
\frac{1}{n} \sum_{j=0}^{n-1} \left \| \Phi_j \left(x_0+\frac{b}{n}\right) \right \|^2 \leq \frac{1}{n} \sum_{j=0}^{n-1} \left \| \Phi_j \left(x_0+\frac{a}{n}\right) \right \|^2 \textrm{exp} \left(2C \alpha_-^{-1} |a-b| \right)
\eeq
and we see that 
\beq \no
\sup_{n} \frac{1}{n} \sum_{j=0}^{n-1} \left \| \Phi_j \left(x_0+\frac{b}{n} \right) \right \|^2 < \infty
\eeq
is bounded uniformly for $b$ in compact sets of $\bbC$.

Now, by Cauchy-Schwarz,
\beq \no
\begin{split}
& \left|\frac{1}{n}\sum_{j=0}^{n-1} q_j\left(x_0+\frac{a}{n} \right)q_j\left(x_0 +\frac{b}{n}\right) \right|  \\
&\quad \leq \left( \frac{1}{n}\sum_{j=0}^{n-1}\left| q_j\left(x_0+\frac{a}{n} \right)\right|^2 \right)^{1/2} \left( \frac{1}{n}\sum_{j=0}^{n-1}\left| q_j\left(x_0+\frac{b}{n} \right)\right|^2  \right)^{1/2}.
\end{split}
\eeq
This implies that for any fixed $a \in \bbC$, the family
\beq \no
\widetilde{g}_n(b)=\frac{1}{n}\sum_{j=0}^{n-1} q_j\left(x_0+\frac{a}{n} \right)q_j\left(x_0 +\frac{b}{n}\right) 
\eeq
is a normal family. Fixing $a \in \bbC \setminus \bbR$, we note that the limit on a set with a limit point determines the limit and so we get that for any fixed  $a \in \bbC \setminus \bbR$ and any $b \in \bbC$
\beq \no
\lim_{n \rightarrow \infty}\frac{\wti{K}_n\left(x_0+\frac{a}{n},x_0+\frac{b}{n} \right)}{n}=\frac{\sin\left(\pi \rho(x_0)(b-a) \right)}{\pi \wti{w}(x_0)(b-a)}.
\eeq
Fixing now $b \in \bbC$ we see that the family
\beq \no
\widetilde{h}_n(a)=\frac{1}{n}\sum_{j=0}^{n-1} q_j\left(x_0+\frac{a}{n} \right)q_j\left(x_0 +\frac{b}{n}\right) 
\eeq
is a normal family, which implies finally that 
\beq \no
\lim_{n \rightarrow \infty}\frac{\wti{K}_n\left(x_0+\frac{a}{n},x_0+\frac{b}{n} \right)}{n}=\frac{\sin\left(\pi \rho(x_0)(b-a) \right)}{\pi \wti{w}(x_0)(b-a)}
\eeq
for any $a, b \in \bbC$, uniformly in compact sets.
\end{proof}

\begin{proof}[Proof of Theorem \ref{thm:rankone}]
By the same arguments as in the proof of Theorem \ref{thm:secondkind}, it is enough to prove \eqref{universality-perturbed} for $a, b$ such that $\Ima(a)>0$, $\Ima(b)<0$. (Note that the second kind polynomials for $\mu^{(\beta_1)}$ are still  $\{q_n\}_{n=0}^\infty$).

Using \eqref{eq:linearcomb}, it is easy to see that 
\beq \label{eq:kernel-linear}
\begin{split}
K_n^{(\beta_1)}\left(x,y \right)& =K_n\left(x,y \right)+\beta_1^2 \wti{K}_n\left(x,y \right)  \\ &\quad-\beta_1\left(\sum_{j=0}^{n-1}q_j \left(x \right) p_j \left(y \right)+\sum_{j=0}^{n-1}q_j \left(y \right) p_j \left(x \right) \right) \\
&=\left(1-\beta_1 \int\frac{d\mu(t)}{x-t}-\beta_1 \int\frac{d\mu(t)}{y-t}  \right)K_n\left(x,y \right)+\beta_1^2 \wti{K}_n\left(x,y \right) \\
&\quad -\beta_1 \left(\int\frac{K_n(x,t)d\mu(t)}{y-t}+\int\frac{K_n(y,t)}{x-t} \right)
\end{split}
\eeq
where the last equality was obtained by substituting $q_j(x)=\int\frac{p_j(x)-p_j(t)}{x-t}d\mu(t)$ and collecting the terms.
Now, all we have to do is compute the appropriate limits for $x=x_0+\frac{a}{n}$ and $y=x_0+\frac{b}{n}$. By Lemma \ref{lemma-rankone},
\beq \no
\begin{split}
&\lim_{n \rightarrow \infty}\frac{1}{n}\left( \int\frac{K_n\left(x_0+\frac{a}{n},t \right)d\mu(t)}{x_0+\frac{b}{n}-t}+\int\frac{K_n\left(x_0+\frac{b}{n},t \right)}{x_0+\frac{a}{n}-t} \right) \\
&\quad=-\left( \int_{-\infty}^\infty\frac{\sin \pi \rho(x_0) (s-a)}{\pi(s-a)(s-b)}ds+\int_{-\infty}^\infty\frac{\sin \pi \rho(x_0) (s-b)}{\pi(s-a)(s-b)}ds \right)=0
\end{split}
\eeq
by \eqref{eq:important_equality}.

Since 
\beq \no
\begin{split}
&\lim_{n \rightarrow \infty} \left(1-\beta_1 \int\frac{d\mu(t)}{x_0+\frac{a}{n}-t}-\beta_1 \int\frac{d\mu(t)}{x_0+\frac{b}{n}-t}  \right)\\
&\quad =\lim_{n \rightarrow \infty} \left(1+\beta_1 \int\frac{d\mu(t)}{t-x_0-\frac{a}{n}}+\beta_1 \int\frac{d\mu(t)}{t-x_0-\frac{b}{n}}  \right)\\
&\quad =1+\beta_1 (F(x_0+i0)+\overline{F(x_0+i0)}) =1+2\beta_1\textrm{Re}(F(x_0+i0)) 
\end{split}
\eeq
we see that
\beq \no
\begin{split}
&\lim_{n \rightarrow \infty}\frac{1}{n}K_n^{(\beta_1)} \left(x_0+\frac{a}{n},x_0+\frac{b}{n} \right) \\
&=\lim_{n \rightarrow \infty} \left(1+2\beta_1 \textrm{Re}(F(x_0+i0))\right) \lim_{n \rightarrow \infty}\frac{1}{n} K_n \left(x_0+\frac{a}{n}, x_0+\frac{b}{n} \right) \\
&\quad +\beta_1^2\lim_{n\rightarrow \infty} \frac{1}{n} \wti{K}_n\left(x_0+\frac{a}{n},x_0+\frac{b}{n} \right) \\
&\quad=\left(1+2\beta_1 \textrm{Re}(F(x_0+i0))+\beta_1^2 \left|F(x_0+i0) \right|^2 \right) \frac{\sin \pi \rho(x_0) (b-a)}{\pi w(x_0) (b-a)}
\end{split}
\eeq
by Theorem \ref{thm:secondkind}. But by \eqref{eq:stieltjesrankoneac1} this is precisely
\beq \no
\frac{\sin \pi \rho(x_0) (b-a)}{\pi w^{(\beta_1)}(x_0) (b-a)}.
\eeq
We are done.
\end{proof}


\section{Proofs of Theorems \ref{StabilityOfSCCDK} and \ref{RandStabilityofSCCDK}}

Both the proof of Theorem \ref{StabilityOfSCCDK} and that of \ref{RandStabilityofSCCDK} are standard applications of variation of parameters methods. In the proofs we give, we explain the connection and then refer to relevant theorems from the literature.
 
\begin{proof}[Proof of Theorem \ref{StabilityOfSCCDK}]
Let $\mu^{(\beta)}$ be the spectral measure of the perturbed Jacobi matrix and write $p^{(\beta)}_k$ and $q^{(\beta)}_k$ for the first and second kind orthogonal polynomials, respectively, associated with the measure $\mu^{(\beta)}$. Theorem 3 and Corollary 1.3 of \cite{als} say that if 
\beq \label{eq:firstkindconvergence}
\lim_{n \rightarrow \infty}\frac{1}{n}\sum_{j=0}^{n-1}|p_j^{(\beta)}(x)|^2
\eeq
and
\beq \label{eq:secondkindconvergence}
\lim_{n \rightarrow \infty}\frac{1}{n}\sum_{j=0}^{n-1}|q_j^{(\beta)}(x)|^2
\eeq
both exist and are finite for a.e.\ $x \in A$ then quasi universality holds for a.e.\ $x \in A$. We shall show that this is indeed the case.

First note that Lebesgue a.e.\ point of $A$ is a strong Lebesgue point of $\mu$. Thus, it follows from the assumptions of the theorem and Theorem \ref{thm:secondkind} that 
\beq \no
\lim_{n \rightarrow \infty}\frac{1}{n}\sum_{j=0}^{n-1}|p_j(x)|^2
\eeq
and
\beq \no
\lim_{n \rightarrow \infty}\frac{1}{n}\sum_{j=0}^{n-1}|q_j(x)|^2
\eeq
both exist and are finite for a.e.\ $x \in A$. We claim also that 
\beq \no
\lim_{n \rightarrow \infty}\frac{1}{n} \sum_{j=0}^{n-1}p_j(x)q_j(x)
\eeq
exists and is finite for a.e.\ $x \in A$. To see this, consider
\beq \label{cute-formula1}
\frac{1}{n}\sum_{j=0}^{n-1}\left(p_j \left(x+\frac{a}{n} \right)q_j\left(x+\frac{b}{n}\right) +p_j \left(x+\frac{b}{n} \right)q_j\left(x+\frac{a}{n}\right)\right)
\eeq
for $a, b$ satisfying $\Ima(a)>0$, $\Ima(b)<0$. By the arguments in the proof of Theorem \ref{thm:rankone}, if $x$ is a strong Lebesgue point of $\mu$ and quasi universality holds at $x$, then this converges, as $n \rightarrow \infty$, to 
\beq \label{cute-formula2}
2\textrm{Re}(F(x+i0))\frac{\sin \pi \rho(x) (b-a)}{\pi w(x) (b-a)}.
\eeq
Using the same normal family argument as in the proof of Theorem \ref{thm:secondkind} we see that this convergence holds for every $a, b \in \bbC$. In particular, taking $a=b=0$, we get
\beq \no
\lim_{n \rightarrow \infty}\frac{1}{n}\sum_{j=0}^{n-1}p_j \left(x \right)q_j\left(x\right)= \textrm{Re}(F(x+i0))\frac{\rho(x)}{w(x) }
\eeq 
for every strong Lebesgue point of $\mu$ where quasi universality holds. In particular, the limit exists and is finite for Lebesgue a.e.\ $x \in A$.

To prove \eqref{eq:firstkindconvergence} and \eqref{eq:secondkindconvergence} we use variation of parameters. We write 
\beq \no
p^{(\beta)}_k(x)=u_{1,k}(x)p_k(x)+u_{2,k}(x)q_k(x),
\eeq
\beq \no
p^{(\beta)}_{k-1}(x)=u_{1,k}(x)p_{k-1}(x)+u_{2,k}(x)q_{k-1}(x),
\eeq
and
\beq \no
q^{(\beta)}_k(x)=v_{1,k}(x)p_k(x)+v_{2,k}q_k(x),
\eeq
\beq \no
q^{(\beta)}_{k-1}(x)=v_{1,k}(x)p_{k-1}(x)+v_{2,k}q_{k-1}(x).
\eeq
Suppose that we know that $u_k(x)=\left(\begin{array}{c} u_{1,k}(x)\\ u_{2,k}(x) \end{array} \right)$ converges to $u(x)=\left(\begin{array}{c} u_1(x) \\ u_2(x) \end{array} \right)$ as $k \rightarrow \infty$, then it is not hard to see that
\beq \no
\frac{1}{n} \sum_{j=0}^{n-1} u_{1,j}(x)^2 p_{j}(x)^2 \rightarrow u_1(x)^2 \lim_{n \rightarrow \infty}\frac{1}{n}\sum_{j=0}^{n-1} p_j(x)^2,
\eeq
and
\beq \no
\frac{1}{n} \sum_{j=0}^{n-1} u_{2,j}(x)^2 q_{j}(x)^2 \rightarrow u_2(x)^2 \lim_{n \rightarrow \infty}\frac{1}{n}\sum_{j=0}^{n-1} q_j(x)^2,
\eeq
as $n \rightarrow \infty$. With a little more work (using Cauchy-Schwartz), it follows that
\beq \no
\frac{1}{n} \sum_{j=0}^{n-1}u_{1,j}(x) u_{2,j}(x)p_j(x) q_{j}(x) \rightarrow u_1(x)u_2(x) \lim_{n \rightarrow \infty}\frac{1}{n}\sum_{j=0}^{n-1}p_j(x) q_j(x).
\eeq
But from writing
\beq \no
p_j^{(\beta)}(x)^2=u_{1,j}(x)^2 p_{j}(x)^2+ u_{2,j}(x)^2 q_{j}(x)^2+2u_{1,j}(x) u_{2,j}(x)p_j(x) q_{j}(x) 
\eeq
we see that the existence of these limits implies \eqref{eq:firstkindconvergence}. Similarly, convergence of $v_k(x)=\left(\begin{array}{c} v_{1,k}(x)\\ v_{2,k}(x) \end{array} \right)$ implies \eqref{eq:secondkindconvergence}. 

Thus, proving the convergence of $u_k(x)$ and $v_k(x)$ for a.e.\ $x \in A$ will prove the theorem. Note that since $\lim_{n \rightarrow \infty}\frac{1}{n}\sum_{j=0}^{n-1}\left( p_j(x)^2+q_j(x)^2\right)<\infty$ at Lebesgue a.e.\ $x \in A$, $A$ (up to a set of zero  Lebesgue measure) is a subset of the essential support of the a.c.\ part of both $\mu$ and $\wti{\mu}$, where $\wti{\mu}$ is the orthogonality measure of the $\{q_n\}_{n=0}^\infty$, (see, e.g., \cite{last-simon-inventiones}). Thus, the restrictions of both $\mu_{\textrm{ac}}$ and $\wti{\mu}_{\textrm{ac}}$ to $A$ are  equivalent to Lebesgue measure (and are also mutually equivalent).
From 
\beq \no
\int \sum_{k=1}^\infty|\beta_k| p^2_k(x) d\mu(x)=\sum_{k=1}^\infty|\beta_k| < \infty
\eeq and
\beq \no 
\int \sum_{k=1}^\infty |\beta_k| q^2_k(x) d\wti{\mu} (x)= \sum_{k=1}^\infty|\beta_k| <\infty,
\eeq
we see that for Lebesgue a.e.\ $x \in A$ 
\beq \no
\sum_{k=1}^\infty\left( |\beta_k|p_k(x)^2+|\beta_k|q_k(x)^2\right) <\infty
\eeq
which also implies that
\beq \no
\sum_{k=1}^\infty |\beta_k||p_k(x)q_k(x)| <\infty
\eeq
for Lebesgue a.e.\ $x \in A$. 
These are precisely the conditions of the discrete version of Theorem 2.2 in \cite{kls} with $f_+=f_-\equiv 1$ (see especially the remarks after the proof and equation (2.9) there). It follows that for a.e.\ $x \in A$, both $u_k(x)$ and $v_k(x)$ converge as $k \rightarrow \infty$ to a finite limit, which finishes the proof.
\end{proof}

\begin{proof}[Proof of Theorem \ref{RandStabilityofSCCDK}]
We use precisely the same strategy where now $u_k(x)$ and $v_k(x)$ are random vectors which we need to show have limits with probability one. Here, the condition \eqref{L2} precisely means that the conditions of Lemma 3.1 from \cite{bl} are satisfied for Lebesgue a.e.\ $x \in A$ (with $f_+\equiv 1$). It follows that for Lebesgue a.e.\ $x \in A$, $u_k(x)$ and $v_k(x)$ converge a.s.\ to a finite limit. An application of Fubini finishes the proof of the theorem. 
\end{proof}


\section{Proof of Theorem  \ref{no-point}}

\begin{proof}[Proof of Theorem \ref{no-point}]

Assume \eqref{universality} holds but \eqref{infinite-limit} does not. Then 
\beq \label{limit}
\lim_{n \rightarrow \infty} K_n(x,x)=\sum_{j=0}^\infty (p_j(x))^2 =C < \infty.
\eeq 
By \eqref{universality}, for any $a \in \bbR$
\beq \no
\lim_{n \rightarrow \infty} K_n\left(x+\frac{a}{n}, x+\frac{a}{n}\right)= \lim_{n \rightarrow \infty} \sum_{j=0}^{n-1} \left( p_j \left(x+\frac{a}{n} \right) \right)^2=C
\eeq
and also,
\beq \label{mixed-limit}
\lim_{n \rightarrow \infty} K_n \left(x+\frac{a}{n},x \right)=C\frac{\sin \left(\pi \rho(x)a \right)}{\pi \rho(x) a}.
\eeq
It follows that 
\beq \label{difference}
\begin{split}
& \lim_{n \rightarrow \infty} \sum_{j=0}^{n-1} \left(p_j \left(x+\frac{a}{n} \right)-p_j \left(x \right) \right)^2 \\
&= \lim_{n \rightarrow \infty} K_n(x,x)+ \lim_{n \rightarrow \infty} K_n \left(x+\frac{a}{n},x+ \frac{a}{n} \right)-\lim_{n \rightarrow \infty} 2K_n \left( x+\frac{a}{n},x \right) \\
&=2C \left(1- \frac{\sin \left(\pi \rho(x)a \right)}{\pi \rho(x) a}\right).
\end{split}
\eeq

We shall use the fact that $C < \infty$ to show that at the same time, for any $a \in \bbR$,
\beq \label{zero-difference}
\lim_{n \rightarrow \infty} \sum_{j=0}^{n-1} \left(p_j \left(x+\frac{a}{n} \right)-p_j \left(x \right) \right)^2=0,
\eeq
contradicting \eqref{difference} and thus proving the theorem.

Fix $a \in \bbR$ and let $\varepsilon>0$. Define $\varepsilon'=\frac{\varepsilon}{5+\sqrt{3}}$. Let $N_0$ be so large that for any $n \geq N_0$ 
\beq \no
\sum_{j=N_0}^n (p_j(x))^2 \leq \left|C-\sum_{j=0}^{N_0-1} (p_j(x))^2 \right| < \varepsilon',
\eeq
and in addition
\beq \no
\left| C-\sum_{j=0}^n \left( p_j \left(x+\frac{a}{n} \right) \right)^2 \right|< \varepsilon'.
\eeq
Now let $N_1$ be so large that for any $n \geq N_1$
\beq \no
\sum_{j=0}^{N_0-1}\left| \left( p_j(x) \right)^2-\left (p_j \left(x+\frac{a}{n} \right) \right)^2 \right|< \varepsilon',
\eeq
and also
\beq \no
\sum_{j=0}^{N_0-1} \left(p_j \left(x+\frac{a}{n} \right)-p_j \left(x \right) \right)^2 <\varepsilon'.
\eeq
(This can clearly be done since $\{p_j\}_{j=0}^{N_0-1}$ is a finite set of continuous functions). It follows that for any $n \geq \max(N_0,N_1)$
\beq \no
\begin{split}
\sum_{j=N_0}^n \left ( p_j \left(x+\frac{a}{n} \right) \right)^2 &=\sum_{j=0}^n \left ( p_j \left(x+\frac{a}{n} \right) \right)^2-C+C-\sum_{j=0}^{N_0-1} \left ( p_j \left(x+\frac{a}{n} \right) \right)^2 \\
& \leq \left | \sum_{j=0}^n \left( p_j \left(x+\frac{a}{n} \right) \right)^2-C \right|+\left|C-\sum_{j=0}^{N_0-1}\left( p_j(x) \right)^2 \right| \\
&\quad+\left|\sum_{j=0}^{N_0-1}\left( p_j(x) \right)^2-  \left ( p_j \left(x+\frac{a}{n} \right) \right)^2\right|<3\varepsilon'.
\end{split}
\eeq

Thus, for any $n \geq \max(N_0,N_1)$
\beq \no
\begin{split}
&\sum_{j=0}^{n}\left(p_j \left(x+\frac{a}{n} \right)-p_j \left(x \right) \right)^2  \leq \sum_{j=N_0}^n\left(p_j \left(x+\frac{a}{n} \right)-p_j \left(x \right) \right)^2+\varepsilon' \\
&=\sum_{j=N_0}^n \left( p_j(x) \right)^2+\sum_{j=N_0}^n \left( p_j \left(x+\frac{a}{n} \right) \right)^2-\sum_{j=N_0}^n p_j(x) p_j \left ( x+\frac{a}{n}\right)+\varepsilon' \\
&<\varepsilon'+3 \varepsilon'+\left(\sum_{j=N_0}^n p_j(x)^2 \right)^{1/2}\left(\sum_{j=N_0}^n \left(p_j \left(x+\frac{a}{n} \right) \right)^2 \right)^{1/2}+\varepsilon' \\
&<(5+\sqrt{3}) \varepsilon'=\varepsilon.
\end{split}
\eeq
We are done.
\end{proof}


\end{document}